\documentclass{article}
\usepackage{amsmath,amsthm,amssymb}
\usepackage{graphicx}
\usepackage{color}
\theoremstyle{plain}
\newtheorem{thm}{Theorem}[section]
\newtheorem{cor}[thm]{Corollary}
\newtheorem{lem}[thm]{Lemma}

\theoremstyle{definition}

\theoremstyle{remark}

\numberwithin{equation}{section}
\newcommand{\z}{z}

\begin{document}
\title{Two-term, asymptotically sharp estimates for eigenvalue means of the Laplacian}%
\author{Evans M. Harrell II and Joachim Stubbe}%
\begin{center}
Two-term, asymptotically sharp estimates for eigenvalue means of the Laplacian

\bigskip
Evans M. Harrell II\\ School of Mathematics, Georgia Institute of Technology\\
 Atlanta GA 30332-0160, USA\\
{\tt harrell@math.gatech.edu}

\bigskip
Joachim Stubbe\\
 EPFL, MATHGEOM-FSB, Station 8, CH-1015 Lausanne, Switzerland\\
{\tt Joachim.Stubbe@epfl.ch}
\end{center}%
%\thanks{}%
%\subjclass{subjects}%
%\keywords{Neumann Laplacian, Dirichlet Laplcian, semiclassical bounds for eigenvalues}

\date{11 June, 2016}%
%\dedicatory{}%
%\commby{}%
% ----------------------------------------------------------------
\begin{abstract}
We present asymptotically sharp inequalities for the eigenvalues
$\mu_k$ of the Laplacian on
a domain with Neumann boundary conditions, using the averaged
variational principle introduced in \cite{HaSt14}.  For the Riesz
mean $R_1(z)$ of the eigenvalues
we improve the known sharp semiclassical bound in terms of the volume of the domain
with a second term with the best possible expected power of $z$.

In addition, we obtain two-sided bounds for individual $\mu_k$, which are semiclassically sharp.  In a final section, we
remark upon the Dirichlet case with the same methods.



\end{abstract}
\maketitle
\setcounter{page}{1}
\tableofcontents
% ----------------------------------------------------------------
\section{Introduction} Let $\Omega\subset\mathbb{R}^d$ be a bounded domain with boundary $\partial\Omega$.
We mainly consider here the eigenvalue problem for the Laplacian with
Neumann boundary conditions,
\begin{equation}\label{Neumann-ev-problem}
\begin{split}
-\Delta u&= \mu u \quad \text{on }\Omega,\\
\frac{\partial u}{\partial n} &= 0 \quad \;\;\text{ on }\partial\Omega.
\end{split}
\end{equation}
The spectrum \eqref{Neumann-ev-problem} consists of an ordered
sequence of eigenvalues $\mu_j$ tending to infinity,
\begin{equation}\label{Neumann-eigenvalues}
  0=\mu_1<\mu_2\leq\mu_3\leq \cdots
\end{equation}
The corresponding normalized eigenfunctions are denoted $u_j$.
Neumann eigenvalues satisfy the same Weyl asymptotic relation as the
better studied Dirichlet eigenvalues, {\em viz}.,
\begin{equation}\label{ev-asymptotics}
  \underset{j\to\infty}{\lim}\mu_jj^{\,-2/d}=C_d|\Omega|^{-2/d},
\end{equation}
where $|\Omega|$ denotes the volume of $\Omega$ and the ``classical constant'' $C_d$ is given by
\begin{equation}\label{classical-constant}
  C_d=(2\pi)^2B_d^{-2/d},
\end{equation}
where $\displaystyle B_d=\frac{\pi^{d/2}}{\Gamma(1+d/2)}$ is
the volume of the $d$-dimensional unit ball.
In 1961, P\'olya showed that
\begin{equation}\label{Neumann-Polya}
  \mu_j\leq C_d|\Omega|^{-2/d}(j-1)^{-2/d}
\end{equation}
for all positive integers $j$ when $\Omega$ is any tiling domain of $\mathbb{R}^d$,
{and the opposite inequality in the Dirichlet case, for which we denote the eigenvalues
$\lambda_j$:
\begin{equation}\label{Dirichlet-Polya}
  \lambda_j \geq C_d|\Omega|^{-2/d}j^{-2/d}.
\end{equation}
His still unproven conjecture is that
these inequalities hold for all bounded domains} $\Omega\subset\mathbb{R}^d$. In other words the Weyl limit \eqref{ev-asymptotics} is approached from below in the Neumann case and above for Dirichlet.

Whereas there are universal relations among eigenvalues of
the Dirichlet problem, for the
Neumann problem,
Colin-de-Verdi\`{e}re showed in 1987 \cite{CdV}
that for any finite nondecreasing $0=\mu_1<\mu_2\leq\mu_3\leq \cdots\leq \mu_k$, there exists a bounded domain having these values as the
first $k$ eigenvalues.
Therefore inequalities among Neumann eigenvalues must incorporate geometric properties of $\Omega$ to be of interest.
(See, e.g.\ \cite{Ash1,AsHe,Ash2} for discussions of universal eigenvalue bounds and related references.)

{Other convenient ways to study the spectrum rely on the counting function,
\begin{equation}\label{counting function}
  \mathcal{N}(\mu):=\sharp\{\mu_j: \mu_j<\mu\},
\end{equation}
and, in a tradition going back to Berezin \cite{Ber}, Riesz means, $R_\sigma(z) :=
\sum_{j}(z-\mu_{j})_{+}^\sigma$, or, resp., $\sum_{j}(z-\lambda_{j})_{+}^\sigma$. {Here $x_{+}$ denotes the positive part of $x$.}
$\mathcal{N}(z)$ can be interpreted as the limit of $R_\sigma(z)$ when $\sigma \to 0$.
For instance, Berezin proved the equivalent of the summed
version of \eqref{Dirichlet-Polya} in the Riesz-mean form,}
\begin{equation}\label{Berezin-ineq}
  \sum_{j}(z-\lambda_{j})_{+}\leq L_{1,d}^{cl}\,|\Omega|z^{1+d/2},
\end{equation}
where
\begin{equation}\label{L-gamma-class}
  L_{\gamma,d}^{cl}:=   \frac{\Gamma(\gamma +1)}{(4\pi)^{d/2}\Gamma(\gamma+1+d/2)}.
\end{equation}
In recent years, beginning with a paper by Melas \cite{Mel}, there has arisen an industry to improve \eqref{Berezin-ineq} by including further terms in lower powers of $z$.  An improvement incorporating
the best expected succeeding power, $z^{d-\frac 1 2}$
was obtained in the Dirichlet case
by Geisinger-Laptev-Weidl \cite{GeLaWe}, and we refer to that paper for further background.

Our main goal here is to achieve analogous improvements in Riesz-means for Neumann eigenvalues
in terms of $z$ to the expected powers.  In addition, we obtain two-sided bounds for individual eigenvalues $\mu_k$, which are semiclassically sharp.  For this we rely on the averaged
variational introduced in \cite{HaSt14} and a series of analytic inequalities.  In a final section, we also treat the Dirichlet case with the same methods.  An appendix contains a discussion of improvements to Young's and H\"older's inequalities.

An important step towards P\'olya's conjecture in the Neumann case
was taken in 1991 by Kr\"{o}ger, who by applying a variational estimate
for the sum of the first $k$ eigenvalues,
obtained the asymptotically sharp inequality
\begin{equation}\label{Kroeger-ev-sum}
  \frac{d+2}{d}\sum_{j=1}^k\mu_j\leq C_d|\Omega|^{-2/d}k^{1-2/d}.
\end{equation}
Later, using the Fourier transforms of the eigenfunctions $u_j$, Laptev proved the
Riesz-mean inequality equivalent to Kr\"{o}ger's estimate \eqref{Kroeger-ev-sum},
\begin{equation}\label{Laptev-Riesz-mean}
  \sum_{j}(z-\mu_{j})_{+}\geq L_{1,d}^{cl}\,|\Omega|z^{1+d/2},
\end{equation}
 for all $z\geq 0$.
(See also \cite{Lap2012}.)
%There it was shown that Kr\"{o}ger's inequality \eqref{Kroeger-sum} holds with a multiplicative,
%non-explicit constant $C(k,d,\Omega)<1$.

Our first result is an improvement of \eqref{Kroeger-ev-sum} using
a refinement of Young's inequality for real numbers, which
not only improves the estimates of Riesz means and sums, but also
provides a bound on individual eigenvalues.
It will be useful to introduce the following notation.
\begin{equation}\label{Kroeger-sum-abbreviations}
  m_k:= C_d\left(\frac{k}{|\Omega|}\right)^{2/d}, \quad S_k:=\frac{\frac{d+2}{d}\frac{1}{k}\sum_{j=1}^k\mu_j}{m_k}.
\end{equation}
In these terms $m_k$ is the Weyl expression, and Kr\"{o}ger's inequality \eqref{Kroeger-ev-sum}
is expressed as $S_k\leq 1$. We shall prove the following refinement of Kr\"{o}ger's inequality.
\begin{thm} Let $d\geq 2$. Then for all $k\geq 0$ the Neumann eigenvalue $\mu_{k+1}$ satisfies
\begin{equation}\label{Neumann-universal-1}
  m_k^2(1-S_k) \geq (\mu_{k+1}-m_k)^2.
\end{equation}
I.e.,
\begin{equation}\label{Neumann-Laplacian-ev-bound-normalized3}
  m_k\left(1-\sqrt{1-S_k}\right)\leq \mu_{k+1} \leq m_k\left(1+\sqrt{1-S_k}\right).
\end{equation}
\end{thm}
Kr\"oger's bound corresponds to replacing the right side of \eqref{Neumann-universal-1} by 0.
One may further ask whether there is an additional remainder term improving the right
side of the universal inequality \eqref{Neumann-universal-1}, which contains more explicit information on the geometry of $\Omega$.
The asymptotic expansion of the counting function suggests that
under sufficient regularity conditions the $(d-1)$-dimensional volume of the boundary $\partial\Omega$ (see \cite{Iv1980,Mel1980}) may appear:
\begin{equation}\label{count-fct-expansion}
  \mathcal{N}(\mu)\approx C_d^{d/2}|\Omega|\mu^{d/2} +\frac1{4}\,C_{d-1}^{(d-1)/2}|\partial\Omega|\mu^{(d-1)/2},
\end{equation}
and therefore, for the Riesz mean,
\begin{equation}\label{R-1-expansion}
 R_1(\z):=\sum_{j=1}(\z-\mu_j)_{+}\approx L_{1,d}^{cl}|\Omega|z^{1+d/2}+\frac1{4}\,L_{1,d-1}^{cl}|\partial\Omega| \z^{(d+1)/2}.
\end{equation}
In the present paper we present a two-term bound for $R_1(\mu)$, using additional geometrical information on $\Omega$.
To this end, for any unit vector ${\bf v}\in \mathbb{R}^d$ let $\delta_{\bf v}$ be the width of $\Omega$ in the ${\bf v}$-direction,
that is,
\begin{equation}\label{v-diameter}
  \delta_{\bf v}(\Omega):=\sup \{{\bf v}\cdot({\bf x}-{\bf y}): {\bf x},{\bf y}\in\Omega\}=\max\{v\cdot({\bf x}-{\bf y}): {\bf x},{\bf y}\in\partial\Omega\}.
\end{equation}
We note that $\delta_{\bf v}(\Omega)$ always lies between twice the inradius and the diameter of $\Omega$. We prove the following.

\begin{thm}\label{Riesz-mean-thm} Let $d\geq 2$. Then for each unit vector ${\bf v}\in \mathbb{R}^d$ and for all $z \geq 0$
\begin{equation}\label{Neumann-Riesz-mean-two-term-bound}
\begin{split}
  \sum(\z-\mu_j)_{+}\geq & L_{1,d}^{cl}|\Omega|\z^{\frac d 2 + 1}+\frac1{4}\,L_{1,d-1}^{cl}\frac{|\Omega|}{\delta_{\bf v}(\Omega)}\,\z^{\frac d 2 + \frac 1 2}\\
&-\frac{1}{96}\,(2\pi)^{2-d}B_{d}\frac{|\Omega|}{\delta_{\bf v}(\Omega)^2}\,\z^{\frac d 2}.\\
\end{split}
\end{equation}
Together with the with the semiclassical bound \eqref{Laptev-Riesz-mean} this implies the improved estimate
\begin{equation}\label{Neumann-Riesz-mean-two-term-bound-bis}
\begin{split}
  \sum(\z-\mu_j)_{+}\geq & L_{1,d}^{cl}|\Omega|\z^{\frac d 2 + 1}\\
&+\Big(\frac1{4}\,L_{1,d-1}^{cl}\frac{|\Omega|}{\delta_{\bf v}(\Omega)}\,\z^{\frac d 2 + \frac 1 2}-\frac{1}{96}\,(2\pi)^{2-d}B_{d}\frac{|\Omega|}{\delta_{\bf v}(\Omega)^2}\,\z^{\frac d 2}\Big)_{+}.\\
\end{split}
\end{equation}
\end{thm}
Both inequalities \eqref{Neumann-Riesz-mean-two-term-bound} and \eqref{Neumann-Riesz-mean-two-term-bound-bis} will follow from our proof.
Although the bound \eqref{Neumann-Riesz-mean-two-term-bound-bis} improves \eqref{Neumann-Riesz-mean-two-term-bound}, we work in most cases with
\eqref{Neumann-Riesz-mean-two-term-bound} since we are mainly interested in large $z$.  An exception is Corollary \ref{corrected Polya cor} below,
where we use the estimate \eqref{Neumann-Riesz-mean-two-term-bound-bis}.
We also remark that while the first term is sharp, the second term in Eq.\eqref{Neumann-Riesz-mean-two-term-bound} appears too small by a factor $1/2$.
Indeed, for the box $\displaystyle \Omega=[0,1]^{d-1}\times [0,\delta]$ the bound \eqref{Neumann-Riesz-mean-two-term-bound} differs from the the asymptotic formula
\eqref{count-fct-expansion} by a factor $1/2$, since with $\delta_{\bf v}(\Omega)=\delta$, comparing the second term of \eqref{Neumann-Riesz-mean-two-term-bound}
and the asymptotic expansion \eqref{R-1-expansion} we obtain
\begin{equation*}
  \frac{|\Omega|}{\delta_{\bf v}(\Omega)}=1,\quad |\partial\Omega|=2+(d-1)\,\delta,
\end{equation*}
in which $\delta$ can be chosen arbitrarily small. More generally, this argument applies to any domain of the form $\displaystyle \Omega=\Omega'\times [0,\delta]$ such that $\Omega'$ is bounded in $\mathbb{R}^{d-1}$ with finite boundary, since
\begin{equation*}
  \frac{|\Omega|}{\delta_{\bf v}(\Omega)}=|\Omega'|,\quad |\partial\Omega|=2|\Omega'|+|\partial\Omega'|\,\delta.
\end{equation*}
From our method of proof it will be seen that for these
kinds of domains the lower bound \eqref{Neumann-Riesz-mean-two-term-bound} can be improved to the
optimal lower bound consistent with the asymptotic formula \eqref{R-1-expansion}.
It is less clear whether the improvement can be obtained
in the absence of a product structure.

\begin{cor} \label{cor-sharp-thm2}
Let $d\geq 2$ and $\displaystyle \Omega=\Omega'\times [0,\delta]$ be a bounded domain.
Then for any unit vector $v\in \mathbb{R}^d$ and for all $\mu\geq 0$,
\begin{equation}\label{Neumann-Riesz-mean-two-term-bound-cart-prod}
\begin{split}
  \sum_{j=1}(\z-\mu_j)_{+}\geq & L_{1,d}^{cl}|\Omega|\z^{\frac d 2 + 1}+\frac1{2}\,L_{1,d-1}^{cl}\frac{|\Omega|}{\delta_{\bf v}(\Omega)}\,\z^{\frac d 2 + \frac 1 2}\\
&-\frac{1}{24}\,(2\pi)^{2-d}B_{d}\frac{|\Omega|}{\delta_{\bf v}(\Omega)^2}\,\z^{\frac d 2}.\\
\end{split}
\end{equation}
\end{cor}

Note that by means of the integral transform
\begin{equation*}
  \int_{0}^{\infty}(z-\lambda-t)_{+}t^{\gamma-2}\,dt,\quad,\gamma>1,
\end{equation*}
Eq. \eqref{Neumann-Riesz-mean-two-term-bound} implies further bounds for higher Riesz means, {\em viz}.,
\begin{equation}\label{Neumann-Riesz-mean-two-term-bound-gamma}
\begin{split}
  \sum_{j=1}(z-\mu_j)_{+}^{\gamma}\geq & L_{\gamma,d}^{cl}|\Omega|z^{\frac{d}{2}+\gamma}+L_{\gamma,d-1}^{cl}\frac{|\Omega|}{4\delta_{\bf v}(\Omega)}\,z^{\frac{d}{2}+\gamma-\frac 1 2}\\
&-\frac{\pi}{96}\,L_{\gamma,d-2}^{cl}\frac{|\Omega|}{\delta_{\bf v}(\Omega)^{2}}\,z^{\frac{d}{2}+\gamma-1},\\
\end{split}
\end{equation}
for any $\gamma \ge 1$,
as well as a strengthened version by means of eq.\eqref{Neumann-Riesz-mean-two-term-bound-bis}.
This moreover implies that P\'olya's conjecture \eqref{Neumann-Polya} can be proved with an improvement for domains in product form:

\begin{cor} \label{corrected Polya cor}

Suppose that $\Omega = \Omega_1 \times \Omega_2$, where P\'olya's conjecture \eqref{Neumann-Polya} holds for $\Omega_1$ and
$\Omega_2$ is a domain of finite measure.  Then
\begin{equation}\label{corrected Polya}
  \mathcal{N}(z)\geq 1 + |\Omega|L_{0,d}^{cl} z^{\frac d 2} + |\Omega|\bigg(\frac{L_{0,d+1}^{cl}}{\sqrt{4 \pi} \cdot 4 \delta_{\bf v}(\Omega_2)} z^{\frac d 2 - \frac 1 2}
         - \frac{L_{0,d+2}^{cl}}{384 \delta_{\bf v}(\Omega_2)^2} z^{\frac d 2 - 1}\bigg)_{+}.
\end{equation}
\end{cor}
\noindent
This implies P\'olya's conjecture for $\Omega$, when only the first two terms in this expression are kept.
\smallskip

The proof of
the main Theorem \ref{Riesz-mean-thm} is based on an averaged variational principle introduced by the authors \cite{HaSt14}, which was later
used in \cite{EHIS}
to extend and simplify Kr\"{o}ger's results for certain operators on manifolds.
The averaged variational principle uses
only basic properties of quadratic forms and an averaging over an orthormal basis or, more generally,
a frame.
Quoting from the formulation in \cite{EHIS}:

\begin{lem}\label{AVP}
Consider a self-adjoint operator $H$ on a Hilbert space $\mathcal{H}$,
the spectrum of which is discrete at least in its lower portion, so that
$- \infty < \mu_0 \le \mu_1 \le \dots$.
The corresponding orthonormalized eigenvectors
are denoted $\{\mathbf{\psi}^{(\ell)}\}$.  The
closed quadratic form corresponding to $H$
is denoted $Q(\varphi, \varphi)$ for vectors $\varphi$ in
the quadratic-form domain $\mathcal{Q}(H) \subset \mathcal{H}$.
Let $f_\zeta \in \mathcal{Q}(H)$ be
a family of
vectors indexed by
a variable $\zeta$ ranging over
a measure space $(\mathfrak{M},\Sigma,\sigma)$.
Suppose that $\mathfrak{M}_0$ is a
subset of $\mathfrak{M}$.  Then for any  $z \in \mathbb{R}$,
\begin{equation}\label{RieszVersion}
	\sum_{j}{\left(z - \mu_j\right)_{+} \int_{\mathfrak{M}}\left|\langle\mathbf{\psi}^{(j)}, f_\zeta\rangle\right|^2\,d \sigma}
       \geq
	\int_{\mathfrak{M}_0}{\left(z\| f_\zeta\|^2 - Q(f_\zeta,f_\zeta) \right) d \sigma},
\end{equation}
provided that the integrals converge.
\end{lem}

\section{Proofs of the main results}
\subsection{Refinement of Kr\"{o}ger's inequality: Theorem 1.1}\label{refinement}
{The quadratic-form domain of the Neumann Laplacian $-\Delta^N$ on a Euclidean domain $\Omega$
is the restriction to $\Omega$ of functions in  the Sobolev space $H_0^1(\mathbb{R}^d)$
\cite{EdEv}
(which is normally but not always the same as
$H^1(\Omega)$)},
and the quadratic form corresponding to $-\Delta^N$ is
\begin{equation}\label{Neumann-quadratic-form}
  Q(f,f)=\int_{\Omega}|\nabla f|^2\,dx.
\end{equation}
The trial functions $\displaystyle f({\bf x})=e^{i{\bf p}\cdot{\bf x}}$ are admissible, so choosing them as in \cite{Kro}
leads after a calculation
to the following bound for the eigenvalues of the Neumann Laplacian.
(The set $\mathfrak{M}$ is chosen as $\{{\bf p} \in \mathbb{R}^d\}$ with Lebesgue measure, and $\mathfrak{M_0}$ is the ball of radius $R$.  See
\cite{Kro,EHIS} for details of the calculation.)
\begin{equation}\label{Neumann-Laplacian-ev-bound}
  \mu_{k+1}R^d-\frac{d}{d+2}R^{d+2}\leq  m_k^{d/2}\left(\mu_{k+1}-\frac{1}{k}\sum_{i=1}^{k}\mu_i\right)
\end{equation}
for all $R>0$, cf.\ \eqref{Kroeger-sum-abbreviations}.
Putting $R^d=m_k^{d/2}x^{d/2}$, we get the bound
\begin{equation*}
  \frac{d+2}{d}\frac{1}{k}\sum_{i=1}^{k}\mu_i\leq m_k\left(\frac{d+2}{d}\frac{\mu_{k+1}}{m_k}-\frac{d+2}{d}\frac{\mu_{k+1}}{m_k}\,x^{\frac{d}{2}}+x^{\frac{d+2}{2}}\right).
\end{equation*}
We choose $\displaystyle x=x_k=\frac{\mu_{k+1}}{m_k}$. This yields
\begin{equation}\label{Neumann-Laplacian-ev-bound-normalized1}
  \frac{d+2}{d}\frac{1}{k}\sum_{i=1}^{k}\mu_i-m_k\leq m_k\frac{2}{d}\,\left(\frac{d+2}{2}\,x_k-\frac{d}{2}-x_k^{\frac{d+2}{2}}\right).
\end{equation}
We may assume that $d\geq2$, since when $d=1$ all eigenvalues are explicitly known. Then
$\displaystyle p=\frac{d}{2}\geq 1$,
and, therefore, the function $g_p(x)$ defined in \eqref{Young-fct-refinement2}
is $\leq 0$.
Hence we obtain:
\begin{equation}\label{Neumann-Laplacian-ev-bound-normalized2}
  \frac{d+2}{d}\frac{1}{k}\sum_{i=1}^{k}\mu_i-m_k\leq -m_k\,(x_k-1)^2,
\end{equation}
which strengthens Kr\"oger's estimate
\begin{equation*}
  \frac{d+2}{d}\frac{1}{k}\sum_{i=1}^{k}\mu_i\leq m_k=C_d\frac{k^{2/d}}{|\Omega|^{2/d}}
\end{equation*}
and yields the bound on $\mu_{k+1}$
%\begin{equation}
%  m_k(1-\sqrt{1-S_k})\leq \mu_{k+1} \leq m_k(1+\sqrt{1-S_k}),
%\end{equation}
claimed in \eqref{Neumann-Laplacian-ev-bound-normalized3}.

\subsection{Two-term spectral bounds: Proof of Theorem 1.2}
\begin{proof}
Let ${\bf v}\in \mathbb{R}^d$
be a unit vector. After a translation we may suppose that $\Omega \subset \mathbb{R}^d$ is a bounded domain such that $\Omega \subset \{{\bf x}\in \mathbb{R}^d: 0\leq {\bf v}\cdot{\bf x}\leq L\}$, that is, in the ${\bf v}$ direction all ${\bf x}\in\Omega$ are contained in an interval of length $L$. We shall choose $L$ later as $L=2\delta_{\bf v}(\Omega)$. Fixing ${\bf v}$, we may choose a coordinate system such that ${\bf v}$ is a standard unit vector of the canonical basis of $\mathbb{R}^d$.
We apply the averaged variational principle \ref{AVP}
with test functions of the form
\begin{equation}\label{Neumann-band-testfct}
  f(x)=(2\pi)^{-\,\frac{d-1}{2}}e^{i{\bf p}_{\bot}\cdot{\bf x}}\phi_n({\bf v}\cdot{\bf x}),
\end{equation}
where ${\bf p}_{\bot}={\bf p}-({\bf p}\cdot {\bf v}){\bf v}$ and $\phi_n$ is an eigenfunction of the Neumann Laplacian on an interval of length $L$, that is,
\begin{equation}\label{+d-Neumann-problem}
  -\phi_n''(y)=\kappa_n\phi_n(y)\quad \text{on}\,]0,L[\, \text{and}\,\phi_n'(0)=\phi_n'(L)=0.
\end{equation}
Recall that the eigenvalues $\kappa_n$ are given by $\displaystyle \kappa_n=\frac{(\pi n)^2}{L^2}$, $n\in \mathbb{N}$ and the (normalized) eigenfunctions are given by $\displaystyle \phi_0(y)=L^{-1/2}$ and $\displaystyle \phi_n(y)=\sqrt{\frac{2}{L}}\,\cos\left(\frac{\pi n y}{L}\right)$,
where $n$ ranges over the positive integers.
With these test functions, the variational principle implies that
\begin{equation}\label{AVP-1d-Neumann-testfct}
\begin{split}
  \sum_{j=1}^k(\z-\mu_j)|\langle f, u_j\rangle|^2 \geq  &(2\pi)^{1-d}(\z-|{\bf p}_{\bot}|^2)\int_{\Omega}\phi_n({\bf v}\cdot{\bf x})^2\\
  &-(2\pi)^{1-d}\int_{\Omega}\phi'_n({\bf v}\cdot{\bf x})^2\\
  \end{split}
\end{equation}
{for any $z \in [\mu_k, \mu_{k+1}]$.}
When $n>0$ we apply the trigonometric identities $\displaystyle\cos^2t=\frac{1+\cos 2t}{2}$ and $\displaystyle\sin^2t=\frac{1-\cos 2t}{2}$ to $\phi_n({\bf v}\cdot{\bf x})^2$ and $\phi_n'({\bf v}\cdot{\bf x})^2$, respectively. Then for all $n\geq 0$, \eqref{AVP-1d-Neumann-testfct} becomes
\begin{equation}\label{AVP-1d-Neumann-testfct-bis}
\begin{split}
  &\sum_{j=1}^k(\z-\mu_j)|\langle f, u_j\rangle|^2 \geq  \,(2\pi)^{1-d}L^{-1}|\Omega|\,(\z-|{\bf p}_{\bot}|^2- \frac{(\pi n)^2}{L^2})\\
  &+(2\pi)^{1-d}L^{-1}\left(\z-|{\bf p}_{\bot}|^2+ \frac{(\pi n)^2}{L^2}\right)(1-\delta_{0,n})\int_{\Omega}\cos\left(\frac{2\pi n {\bf v}\cdot{\bf x}}{L}\right),\\
  \end{split}
\end{equation}
where $\delta_{0,n}$ denotes the Kronecker delta.
On the right side we integrate over the set
$\displaystyle \Phi_{k}=\{({\bf p}_{\bot},n)\in \mathbb{R}^{d-1}\times \mathbb{N}: |{\bf p}_{\bot}|^2+\frac{\pi^2 n^2}{L^2}\leq \z \}$ while on the left
side over the larger set
$\mathbb{R}^{d-1}\times \mathbb{N}$, using Parseval's identity.
We shall prove in Lemma \ref{1d-Neumann-Laplacian_Riesz-mean-lemma} below that for all $R>0$,
\begin{equation}\label{1d-Neumann-Laplacian_Riesz-mean}
  \sum_{k\geq 0} (R^2-k^2)_{+}\geq \max\left(\frac{2R^3}{3}+\frac{R^2}{2}-\frac{R}{6},R^2\right).
\end{equation}
By  applying the lower bound
\eqref{1d-Neumann-Laplacian_Riesz-mean} to the sum over $n$ and then integrating over ${\bf p}_{\bot}$
we obtain an explicit lower bound for
$\displaystyle\underset{\Phi_k}{\int\sum}\,\left(\z-|{\bf p}_{\bot}|^2- \frac{(\pi n)^2}{L^2}\right)$.
Since $\int\max(f,g)\geq \max(\int f,\int g)$, this yields
\begin{equation}\label{AVP-1d-Neumann-testfct-av}
\begin{split}
  \sum_{j=1}^k(\z-\mu_j)\geq  &\frac{2}{d+2}\,(2\pi)^{-d}B_d|\Omega|\,\z^{\frac{d}{2}+1}\\
  &+\frac{1}{d+1}\,(2\pi)^{1-d}B_{d-1}|\Omega|L^{-1}\,\z^{\frac{d+1}{2}}\\
 &-\frac{1}{24}\,(2\pi)^{2-d}B_{d}|\Omega|L^{-2}\z^{\frac{d}{2}}+G(\z),
  \end{split}
\end{equation}
where
\begin{equation*}
  G(\z):=\underset{\Phi_k}{\int\sum}(2\pi)^{1-d}(\z-|{\bf p}_{\bot}|^2+ \frac{(\pi n)^2}{L^2})(1-\delta_{0,n})\int_{\Omega}\cos\left(\frac{2\pi n {\bf v}\cdot{\bf x}}{L}\right).
\end{equation*}
It remains to control $G(\z)$,
{which could in principle be positive or negative.  In fact, by averaging
\eqref{AVP-1d-Neumann-testfct-av} in a certain way we shall show that $G$ can be dropped
altogether.  To this
end we choose $L$}
large enough that $\Omega$ is also contained in $\{{\bf x}\in \mathbb{R}^d: 0\leq {\bf v}\cdot {\bf x}\leq L\}$
when translated by $L/2$. This means nothing else than assuming that
$\Omega \subset \{{\bf x}\in \mathbb{R}^d: 0\leq {\bf v}\cdot {\bf x} \leq L/2\}$. In the corresponding Neumann eigenfunctions we have to replace ${\bf v}\cdot{\bf x}$ by ${\bf v}\cdot{\bf x}+L/2$. We may apply the averaged variational principle on both sets (the eigenvalues $\displaystyle \frac{(\pi n)^2}{L^2}$, $n\in \mathbb{N}$ remain unchanged). Since
\begin{equation*}
  \frac1{2}\bigg(\cos\left(\frac{2\pi n {\bf v}\cdot{\bf x}}{L}\right)+\cos\left(\frac{2\pi n ({\bf v}\cdot{\bf x}+L/2)}{L}\right)\bigg)= \begin{cases}
    \cos\left(\frac{2\pi n {\bf v}\cdot{\bf x}}{L}\right) & \text{if $n$ is even}, \\
    0 & \text{if $n$ is odd},
  \end{cases}
\end{equation*}
all odd $n$
{may be dropped from $G(\z)$, leaving only cosine functions of the form $\displaystyle cos\left(\frac{4\pi n {\bf v}\cdot{\bf x}}{L}\right)$ with $n$ a positive integer.
We apply the same averaging procedure with a translation by $L/4$.
Since
\begin{equation*}
  \frac1{2}\bigg(\cos\left(\frac{4\pi n {\bf v}\cdot{\bf x}}{L}\right)+\cos\left(\frac{2\pi n ({\bf v}\cdot{\bf x}+L/4)}{L}\right)\bigg)= \begin{cases}
    \cos\left(\frac{4\pi n {\bf v}\cdot{\bf x}}{L}\right) & \text{if $n$ is even}, \\
    0 & \text{if $n$ is odd},
  \end{cases}
\end{equation*}
again the terms containing odd integers may be dropped.
Since $G(\z)$ contains only a finite number of contributions, after a finite sequence of
averages with shifts $L/2^n$, every contribution will be eliminated.}
Hence
\begin{equation}\label{AVP-1d-Neumann-testfct-av-final bound}
\begin{split}
  \sum_{j=1}^k(z-\mu_j)&\geq  \frac{2}{d+2}\,(2\pi)^{-d}B_d|\Omega|\,z^{\frac{d}{2}+1}\\
  &\quad +\frac{1}{d+1}\,(2\pi)^{1-d}B_{d-1}|\Omega|L^{-1}\,z^{\frac{d+1}{2}}\\
 &\quad -\frac{1}{24}\,(2\pi)^{2-d}B_{d}|\Omega|L^{-2}z^{\frac{d}{2}}.
  \end{split}
\end{equation}
We may now choose $L=2\delta_{\bf v}(\Omega)$, which yields the statement of the theorem.
\end{proof}
\paragraph{} To prove Corollary \ref{cor-sharp-thm2} we note that when
$\displaystyle \Omega=\Omega'\times [0,\delta]$ we
may choose ${\bf v}={\bf e}_n$. As a consequence
\begin{equation*}
  \int_{\Omega}\phi_n({\bf v}\cdot{\bf x})^2=|\Omega'|,\quad \int_{\Omega}\phi'_n({\bf v}\cdot{\bf x})^2=\frac{(\pi n)^2}{L^2}\,|\Omega'|,
\end{equation*}
and no translations are needed. Therefore we may choose $L=\delta$ which yields the bound \eqref{Neumann-Riesz-mean-two-term-bound-cart-prod}.

\noindent
From the bound \eqref{1d-Neumann-Laplacian_Riesz-mean} it is straightforward
to derive the simpler expression
\begin{equation}\label{1d-Neumann-Laplacian_Riesz-mean-weaker-lb}
  \sum_{k\geq 0} (R^2-k^2)_{+}\geq \frac{2R^3}{3}+\frac{R^2}{3},
\end{equation}
containing only two terms. This yields the following spectral bound.
\begin{cor}\label{2termcor}
Let $d\geq 2$. Then for any unit vector $v\in \mathbb{R}^d$ and for all $\mu\geq 0$
\begin{equation}\label{Neumann-Riesz-mean-weaker-two-term-bound}
  \sum_{j=1}^k(z-\mu_j)_{+}\geq  L_{1,d}^{cl}|\Omega|z^{\frac d 2 +1}+
 {L_{1,d-1}^{cl}\frac{|\Omega|}{6\delta_{\bf v}(\Omega)}}\,z^{\frac d 2 +\frac 1 2}.
\end{equation}
\end{cor}

%which is defined for all $\gamma\geq 0$ and, in particular, $\displaystyle L_{0,d}^{cl}=C_d$.
%TO DISCUSS WHETHER TO PRESENT HERE OR NOT

The term containing the width
$\delta_{\bf v}$ can be estimated by geometric properties of the convex hull of
$\Omega$, since $\delta_{\bf v}(\Omega)$ coincides with
$\delta_{\bf v}({\rm hull}(\Omega))$.
For example, in 2 dimensions,
\begin{equation}\label{perim}
\int_{S^{1}}{\delta_{\bf v}} = 2 |\partial {\rm hull}(\Omega)|.
\end{equation}
With Corollary \ref{2termcor}, by choosing ${\bf v}$ so that $\delta_{\bf v}$ equals the
mean width $w$ of $ {\rm hull}(\Omega)$
(= the average of $\delta_{\bf v}$
uniformly over directions ${\bf v}$),
we obtain a correction involving the isoperimetric ratio of $\Omega$,
\begin{equation}\label{hullperim}
  \sum_{j=1}^k(\mu-\mu_j)_{+} \geq  L_{1,2}^{cl}|\Omega|\mu^{2}+L_{1,1}^{cl}\frac{\pi |\Omega|}{6|\partial {\rm hull}(\Omega)|}\,\mu^{3/2}.
\end{equation}
In arbitrary dimensions, if
$\delta_{\bf v}$ is chosen equal to $w$,
then, following Bourgain \cite{Bou},
the final term in \eqref{Neumann-Riesz-mean-weaker-two-term-bound}
%\item  $w \ge \left(\frac{|\Omega|}{\omega_d}\right)^{\frac 1 d}$  (Urysohn's inequality)
can be bounded from below in terms of the
{\em isotropic constant},
$$
L_{{\rm hull}(\Omega)|}^2 := \frac{\det(M_{{\rm hull}(\Omega)|})^{\frac{1}{d}}}{{\rm Vol({\rm hull}(\Omega)|)^{1+\frac{2}{d}}}},
$$
where the inertia matrix $M_{ij} = \int_{{\rm hull}(\Omega)|}{x_i x_j dx}$ has been minimized with respect to the choice of the origin.
Finding the optimal upper bound for the ratio $\frac{w}{L_\Omega}$ for convex $\Omega$
is an open problem in analysis.  In \cite{Mil}, Milman has, for example, proved an upper bound for $w$
in the form of a universal
constant times $\sqrt{d}\, \log(d)^2$.

It has been known since the work of Ball \cite{Bal} that under various
further assumptions convex bodies satisfy reverse isoperimetric inequalities,
with which Inequality \eqref{hullperim} can be connected to additional geometric properties of ${\rm hull}(\Omega)$.  See, e.g., \cite{PaZh}.

\bigskip

Finally, we prove Corollary \ref{corrected Polya cor}.  It was shown by Laptev \cite{Lap1997} that
if P\'olya's conjecture holds on a domain $\Omega_1$, then it holds on arbitrary Cartesian products
of the form $\Omega_1 \times \Omega_2$.  (The Dirichlet case was treated in \cite{Lap1997}.)  In fact, the same argument
allows improve bounds on the counting function, benefitting from the improved bounds for sums coming from
$\Omega_2$, as follows.

\begin{proof}

Suppose that $\Omega_1\subset \mathbb{R}^{d_1}$ $d_1\geq 2$ is a domain for which P\'olya's conjecture
\begin{equation*}
  \mathcal{N}(z)=\sum_{j}(z-\mu_{j})_{+}^{0}\geq 1+L_{0,d_1}^{cl}\,|\Omega_1|z^{d_1/2}
\end{equation*}
is valid.
Let $d_1+d_2=d$, $\Omega=\Omega_1\times\Omega_2$ with $\Omega_2\subset \mathbb{R}^{d_2}$ of finite measure. The Neumann eigenvalues $\mu_j$ of $\Omega$ are of the form $\mu_j=\mu_{j_1}+\mu_{j_2}$ where $\mu_{j_1},\mu_{j_2}$ are the Neumann eigenvalues of $\Omega_1,\Omega_2$, respectively.
Therefore,
\begin{equation*}
  \sum_{j}(z-\mu_{j})_{+}^{0}=\sum_{j_2}\sum_{j_1}(z-\mu_{j_2}-\mu_{j_1})_{+}^{0}\geq 1+L_{0,d_1}^{cl}\,|\Omega_1|\sum_{j_2}(z-\mu_{j_2})_{+}^{d_1/2}.
\end{equation*}
Since $d_1/2\geq 1$, using \eqref{Neumann-Riesz-mean-two-term-bound-gamma} and \eqref{L-gamma-class} we obtain
\begin{align*}
  \mathcal{N}(z) &\geq 1+ L_{0,d_1}^{cl}\,|\Omega_1|\,|\Omega_2|\bigg(L_{\frac{d_1}{2},d_2}^{cl} z^{\frac d 2} +
  \frac{L_{\frac{d_1}{2},d_2 -2}^{cl}}{4 \delta_{\bf v}(\Omega_2)} z^{\frac d 2 - \frac 1 2}\\
         &\quad\quad -\frac{\pi}{96}\,\frac{L_{\gamma,d-2}^{cl}}{\delta_{\bf v}(\Omega_2)^{2}}\,z^{\frac{d}{2}+\gamma-1}\bigg) \\
         &\geq 1 + |\Omega|\bigg(L_{0,d}^{cl} z^{\frac d 2} + \frac{L_{0,d+1}^{cl}}{\sqrt{4 \pi} \cdot 4 \delta_{\bf v}(\Omega_2)} z^{\frac d 2 - \frac 1 2}
         - \frac{L_{0,d+2}^{cl}}{384 \delta_{\bf v}(\Omega_2)^2} z^{\frac d 2 - 1}\bigg),
\end{align*}
as well as
\begin{equation*}
  \mathcal{N}(z)\geq 1 + |\Omega|L_{0,d}^{cl} z^{\frac d 2}.
\end{equation*}
Combining both estimates we prove the claim.

\end{proof}

\section{Riesz means of Laplacians on rectangles}
In this section we derive upper and lower bounds for Riesz means of Neumann and Dirichlet Laplacians, respectively,
on the rectangle $R:=[0,l_1]\times[0,l_2]$.
\begin{thm}\label{thm-ev-rectangle}
Let $\mu_i^R,\lambda_i^R$ denote the eigenvalues of the Neumann Laplacian and the Dirichlet Laplacian on $R=[0,l_1]\times[0,l_2]$. Suppose that $l_1\leq l_2$. Then the following estimates hold
\begin{equation}\label{R-1-rectangle-Neumann-bounds}
 \begin{split}
&\frac{3\pi}{128}\left(\frac{l_2}{l_1}+\frac{l_1}{l_2}+\frac{32}{3\pi}\right)\,\mu+\frac{3\pi}{64}\left(\frac{1}{l_1}+\frac{1}{l_2}\right)\,\mu^{1/2}
+\frac{3\pi^{3/2}2^{1/2}}{64 l_2l_1^{1/2}}\,\mu^{1/4}\\
&\geq \sum_{j=1}^k(\mu-\mu_j^R)_{+}-\frac{|R|}{8\pi}\,\mu^2-\frac{|\partial R|}{6\pi}\,\mu^{3/2}\\
&\geq -\frac{\pi}{24}\left(\frac{l_2}{l_1}+\frac{l_1}{l_2}-\frac{6}{\pi}\right)\,\mu-\frac{\pi}{12}\left(\frac{1}{l_1}+\frac{1}{l_2}\right)\,\mu^{1/2}-\frac{\pi^{3/2}2^{1/2}}{12 l_2l_1^{1/2}}\,\mu^{1/4},
\end{split}
\end{equation}
and
\begin{equation}\label{R-1-rectangle-Dirichlet-bounds}
 \begin{split}
&\frac{3\pi}{128}\left(\frac{l_2}{l_1}+\frac{l_1}{l_2}+\frac{32}{3\pi}\right)\,\lambda+\frac{\pi}{12}\left(\frac{1}{l_2}-\frac{9}{16l_1}\right)\,\lambda^{1/2}
+\frac{3\pi^{3/2}2^{1/2}}{64 l_2l_1^{1/2}}\,\lambda^{1/4}\\
&\geq \sum_{j=1}^k(\lambda-\lambda_j^R)_{+}-\frac{|R|}{8\pi}\,\lambda^2+\frac{|\partial R|}{6\pi}\,\lambda^{3/2}\\
&\geq -\frac{\pi}{24}\left(\frac{l_2}{l_1}+\frac{l_1}{l_2}-\frac{6}{\pi}\right)\,\lambda+\frac{\pi}{12}\left(\frac{1}{l_1}-\frac{9}{16l_2}\right)\,\lambda^{1/2}-\frac{\pi^{3/2}2^{1/2}}{12 l_2l_1^{1/2}}\,\lambda^{1/4}.
\end{split}
\end{equation}

\end{thm}
\begin{proof}
The Riesz mean for the Neumann Laplacian on $R$ is given by
\begin{equation*}
  R_1^N(z)=\underset{n_1,n_2\geq 0}{\sum\sum}\bigg(z-\frac{(\pi n_1)^2}{l_1^2}-\frac{(\pi n_2)^2}{l_2^2}\bigg)_{+}.
\end{equation*}
We need the following polynomial upper and lower bounds for
one-dimensional Riesz means $\sum(R^2 - k^2)_+^p$, in particular \eqref{1d-Neumann-Laplacian_Riesz-mean}.

\begin{lem}\label{1d-Neumann-Laplacian_Riesz-mean-lemma} For all $R>0$,
\begin{equation}\label{1d-Neumann-Laplacian_Riesz-mean-bis}
  \max\left(\frac{2R^3}{3}+\frac{R^2}{2}-\frac{R}{6},R^2\right) \leq
  \sum_{k\geq 0} (R^2-k^2)_{+}\leq \frac{2R^3}{3}+\frac{R^2}{2}+\frac{3 R}{32},
\end{equation}
and for all $R>0$, $\beta> 0$,
\begin{equation}\label{beta-Riesz-mean-lower-bound}
\begin{split}
  &\max\left(\frac{\sqrt{\pi}\,\Gamma(\beta+2)}{2\,\Gamma(\beta+5/2)}\,R^{2\beta+3}+\frac1{2}\,R^{2\beta+2}
  -\frac{\sqrt{\pi}\,\Gamma(\beta+2)}{12\,\Gamma(\beta+3/2)}\,R^{2\beta+1},R^{2\beta+2}\right)\\
  &\leq\\
  &\sum_{k\geq 0}(R^2-k^2)_{+}^{\beta+1}\\
  &\leq\\
  & \frac{\sqrt{\pi}\,\Gamma(\beta+2)}{2\,\Gamma(\beta+5/2)}\,R^{2\beta+3}+\frac1{2}\,R^{2\beta+2}
  +\frac{3\sqrt{\pi}\,\Gamma(\beta+2)}{64\,\Gamma(\beta+3/2)}\,R^{2\beta+1}.
   \end{split}
\end{equation}
Finally, for all $R>0$,
\begin{equation}\label{sqrt-Riesz-mean-upper-bound}
   \sum_{k\geq 0} \sqrt{(R^2-k^2)_{+}}\leq \frac{\pi R^2}{4}+\frac{R}{2}+\frac{\sqrt{2R}}{2}.
\end{equation}

\end{lem}

The lemma will be proved below.  Assuming it for now, we continue
the proof of the theorem for the Neumann Laplacian on the rectangle $[0,l_1]\times [0,l_2]$.
Since
\begin{equation*}
  R_1^N(z)=\frac{\pi^2}{l_2^2}\underset{n_1,n_2\geq 0}{\sum\sum}\bigg(\frac{l_2^2z}{\pi^2}-\frac{l_2^2 n_1^2}{l_1^2}- n_2^2\bigg)_{+},
\end{equation*}
by applying the lower bound \eqref{1d-Neumann-Laplacian_Riesz-mean-bis} we get:
\begin{equation*}
\begin{split}
   R_1^N(z)\geq &\frac{2\pi^2l_2}{3l_1^3}\sum_{n_1\geq 0}\bigg(\frac{l_1^2z}{\pi^2}- n_1^2\bigg)^{3/2}_{+}+\frac{\pi^2}{2l_1^2}\sum_{n_1\geq 0}\bigg(\frac{l_1^2z}{\pi^2}- n_1^2\bigg)_{+}\\
&-\frac{\pi^2}{6l_1l_2}\sum_{n_1\geq 0}\bigg(\frac{l_1^2z}{\pi^2}- n_1^2\bigg)^{1/2}_{+}.
\end{split}
\end{equation*}
Applying the lower bounds \eqref{1d-Neumann-Laplacian_Riesz-mean-bis},\eqref{beta-Riesz-mean-lower-bound} and the upper bound \eqref{sqrt-Riesz-mean-upper-bound} we get
\begin{equation*}
  \frac{2\pi^2l_2}{3l_1^3}\sum_{n_1\geq 0}\bigg(\frac{l_1^2z}{\pi^2}- n_1^2\bigg)^{3/2}_{+}\geq\frac{l_1l_2}{8\pi}\,z^2+\frac{l_2}{3\pi}\,z^{3/2}-\frac{\pi}{24}\,\frac{l_2}{l_1}\,z,
\end{equation*}
\begin{equation*}
  \frac{\pi^2}{2l_1^2}\sum_{n_1\geq 0}\bigg(\frac{l_1^2z}{\pi^2}- n_1^2\bigg)_{+}\geq \frac{l_1}{3\pi}\,z^{3/2}+\frac{z}{4}- \frac{\pi}{12 l_1}\,z^{1/2}.
\end{equation*}
and
\begin{equation*}
 -\frac{\pi^2}{6l_1l_2}\sum_{n_1\geq 0}\bigg(\frac{l_1^2z}{\pi^2}- n_1^2\bigg)^{1/2}_{+}\geq -\frac{\pi}{24}\,\frac{l_1}{l_2}\,z-\frac{\pi}{12 l_2}\,z^{1/2}-
\frac{\pi^{3/2}2^{1/2}}{12 l_2l_1^{1/2}}\,z^{1/4}.
\end{equation*}
Summarizing all estimates, we get the lower bound of \eqref{R-1-rectangle-Neumann-bounds}. Similarly, we get the upper bound of \eqref{R-1-rectangle-Neumann-bounds} interchanging $l_1$ and $l_2$. The Riesz mean for the Dirichlet Laplacian on $R$ is given by
\begin{equation*}
  R_1^D(z)=\frac{\pi^2}{l_2^2}\underset{n_1,n_2\geq 1}{\sum\sum}\bigg(\frac{l_2^2z}{\pi^2}-\frac{l_2^2 n_1^2}{l_1^2}- n_2^2\bigg)_{+}.
\end{equation*}
The corresponding one-dimensional bounds are those of Lemma \ref{1d-Neumann-Laplacian_Riesz-mean-lemma} subtracting $R^2$, $R^{2\beta +2}$,
and respectively $R$ in \eqref{1d-Neumann-Laplacian_Riesz-mean-bis}, \eqref{beta-Riesz-mean-lower-bound} and the upper bound
\eqref{sqrt-Riesz-mean-upper-bound}, leading to a change of the sign of the second term, from which we get the bounds \eqref{R-1-rectangle-Dirichlet-bounds} of the theorem.
\end{proof}

We next prove Lemma \ref{1d-Neumann-Laplacian_Riesz-mean-lemma}:
\begin{proof}
Start from the identity
\begin{equation}
  \sum_{k\geq 0} (R^2-k^2)_{+}=R^2+R^2[R]-\frac{[R]^3}{3}-\frac{[R]^2}{2}-\frac{[R]}{6},
\end{equation}
where $[R]$ denotes the integer part of $R$. We substitute the periodic sawtooth function
$\displaystyle \psi(t)=\left(t-[t]-\frac{1}{2}\right)$, in terms of which
\begin{equation}\label{Riesz_psi}
\begin{split}
  \sum_{k\geq 0} (R^2-k^2)_{+}&=\frac{2R^3}{3}+\frac{R^2}{2}-\frac{R}{6}+\left(\frac{1}{4}-\psi(R)^2\right)\left(R-\frac{\psi(R)}{3}\right)\\
&\geq \frac{2R^3}{3}+\frac{R^2}{2}-\frac{R}{6},
\end{split}
\end{equation}
since both factors of the product are nonnegative. This lower bound is exact when $R$ is an integer. Since $\displaystyle \sum_{k\geq 0} (R^2-k^2)_{+}=R^2$
trivially for all $0<R<1$, the lower bound follows. For the upper bound, we wish to replace $\left(\frac{1}{4}-\psi(R)^2\right)\left(R-\frac{\psi(R)}{3}\right)$ by a linear
expression in $R$ for $R\ge0$, or, equivalently, find an upper bound for
$$
F(R) := \left(\frac 1 4 - \psi(R)^2\right)\left(1 - \frac{\psi(R)}{3 R}\right).
$$
Because on each interval $(n, n+1)$ the function $\psi(R)$ is antisymmetric about $n + \frac 1 2$ and
negative on $(n, n+\frac 1 2 )$, the maximum is to be sought in an interval of the form
$(n, n+\frac 1 2 )$.  On these subintervals, the second factor decreases when $R$ is replaced by
$R+1$, while the first factor is positive and unchanged.  Hence, the maximum of
$F(R)$ occurs where $0 < R < \frac 1 2$.
In this interval, however, an elementary calculus exercise shows that the maximizing value is
$R = \frac 3 8$, and thus $F(R) \le F(\frac 3 8 ) = \frac{25}{96}$.  Substituting this into the
first line of \eqref{Riesz_psi} yields the claim. We observe that the upper and lower bounds in \eqref{1d-Neumann-Laplacian_Riesz-mean-bis} coincide uniquely when $R= \frac 3 8$.
\noindent
To prove \eqref{beta-Riesz-mean-lower-bound} we note that for all $\beta>0$,
\begin{equation*}
\begin{split}
   \int_0^{\infty} \sum_{k\geq 0}(R^2-t-k^2)_{+}t^{\beta-1}\,dt&=\frac1{\beta(\beta+1)}\sum_{k\geq 0}(R^2-k^2)_{+}^{\beta+1}\\
   &=2\int_0^{\infty}\sum_{k\geq 0}(s^2-k^2)_{+}\,s\,(R^2-s^2)_{+}^{\beta-1}\;ds,
   \end{split}
\end{equation*}
and then apply the bounds \eqref{1d-Neumann-Laplacian_Riesz-mean-bis}.
\noindent
It remains to show \eqref{sqrt-Riesz-mean-upper-bound}. We start from the identity
\begin{equation*}
  \sum_{k\geq 0} \sqrt{(R^2-k^2)_{+}}=\frac{\pi R^2}{4}+\frac{R}{2}-\int_0^Rt(R^2-t^2)^{-1/2}
\left(t-[t]-\frac1{2}\right)\;dt.
\end{equation*}
For any continuous increasing function $f:[0,R]\rightarrow \mathbb{R}$ and any positive integer $k\leq R$,
\begin{equation}\label{psi-increasing-fct-ineq-1}
  \int_{k-1}^k\psi(t)f(t)\,dt=\int_{0}^{\frac1{2}}\left(\frac1{2}-s\right)\left(f(k-s)-f(k-1+s)\right)\,ds\geq 0.
\end{equation}
Consequently,
\begin{equation}
  \sum_{k\geq 0} \sqrt{(R^2-k^2)_{+}}\leq\frac{\pi R^2}{4}+\frac{R}{2}-\int_{[R]}^Rt(R^2-t^2)^{-1/2}
\left(t-[t]-\frac1{2}\right)\;dt.
\end{equation}
The integral between $[R]$ and $R$ can also be computed explicitly. Define $\displaystyle \rho=\frac{[R]}{R}$ and $\displaystyle \kappa=\sqrt{1-\rho^2}$. Then for all $R>0$ we have $\displaystyle 1-\min\left(1,\frac{1}{R}\right)
\leq \rho\leq 1$.
Hence $\displaystyle 0<\kappa < 1$ if $R<1$ and $\displaystyle 0<\kappa < R^{-1}\sqrt{2R-1}$ otherwise. Then
\begin{equation*}
  \begin{split}
  \int_{[R]}^R\frac{t\,\psi(t)}{\sqrt{R^2-t^2}}\;dt&=R^2\int_{\rho}^1\frac{s^2-\rho s-\frac{s}{2R}}{\sqrt{1-s^2}}\,ds\\
  &=\frac{R^2}{2}\bigg(\arcsin \kappa -\kappa\sqrt{1-\kappa^2}\bigg)-\frac{R}{2}\kappa.
\end{split}
\end{equation*}
We also note that $\displaystyle \kappa\mapsto\arcsin \kappa -\kappa\sqrt{1-\kappa^2}-\frac{2\kappa^3}{3}$ is increasing. It follows that
\begin{equation*}
  \int_{[R]}^R\frac{t\,\psi(t)}{\sqrt{R^2-t^2}}\;dt\leq-\frac{R^2\kappa^3}{3}+\frac{R\kappa}{2},
\end{equation*}
proving the claim.
\end{proof}

\section{Two-term estimates for Dirichlet Laplacians by averaging}
{For Dirichlet Laplacians on a bounded domain $\Omega$ our strategy will be to
enclose $\Omega$ in a box $B$ and then to use the averaged variational principle to
estimate the Riesz means of the Dirichlet Laplacian on $B$ in terms of expectations
with the eigenfunctions of $-\Delta_{\Omega}$.  Thus
suppose that $\Omega\subset B$ where $\displaystyle B=\prod_{\alpha=1}^d]0,L_{\alpha}[$ is a box of volume $\displaystyle |B|=\prod_{\alpha=1}^dL_{\alpha}$.
We let $v_k^{\Omega}$ denote the Dirichlet eigenfunctions on $\Omega$,
and, similarly, for $B$ we define
\begin{equation*}
 v_k^B(x)=\prod_{\alpha=1}^d\psi_{n_{\alpha}}(x_{\alpha}),
\end{equation*}
where $\psi_{n_{\alpha}}(x_{\alpha}) := \sqrt{\frac{2}{l_\alpha}} \sin\left(\frac{n \pi x_\alpha}{l_\alpha}\right)$,
corresponding to eigenvalues $\displaystyle \lambda_k^B=\prod_{\alpha=1}^d\frac{\pi^2n_{\alpha}^2}{l_{\alpha}^2}$ with $n_{\alpha}\in\mathbb{Z}_{+}$.}
By the variational principle,
\begin{equation}
  \sum (z-\lambda_j^{B})_{+}\big|\langle v_k^{\Omega},v_j^B\rangle_{B}\big|^2\geq z\int_{B}|v_k^{\Omega}|^2\,dx-\int_{B}|\nabla v_k^{\Omega}|^2\,dx.
\end{equation}
Since $v_k^{\Omega}\in H_0^1(\Omega)$, all integrals reduce to integrals on $\Omega$.
On the right side we take a finite sum in $k$ while on the left we
sum over all $k$ and apply the completeness relation, obtaining
\begin{equation}\label{AVP-Dirichlet-Omega-Box-1}
  \sum (z-\lambda_j^{B})_{+}\int_{\Omega}|v_j^B(x)|^2\,dx\geq \sum (z-\lambda_j^{\Omega})_{+}.
\end{equation}
To apply the translation argument as above we suppose that $l_\alpha$ is at least twice the
width of $\Omega$ in the $\alpha$ direction and note that the average of
$\psi_{n_\alpha}^2(x_\alpha)$ and its translate by $\frac{l_\alpha}{2}$ is $\frac{1}{l_\alpha}$.
Therefore,
\begin{equation}\label{AVP-Dirichlet-Omega-Box-2}
  \frac{|\Omega|}{|B|}\sum (z-\lambda_j^{B})_{+}\geq \sum (z-\lambda_j^{\Omega})_{+},
\end{equation}
which improves Berezin-Li-Yau.
Consider, for example, the case $d=2$ where applying the upper bound in \eqref{R-1-rectangle-Dirichlet-bounds} of theorem \ref{thm-ev-rectangle} for the Dirichlet Laplacian on a rectangle $B$ with side lengths $l_1,l_2$ we obtain the explicit upper bound
\begin{equation}\label{Dirichlet-Box-2d-lower bound}
  \sum (\lambda-\lambda_j^{\Omega})_{+}\leq L_{1,2}^{cl}|\Omega|\,\lambda^2-\frac1{4}\,L_{1,1}^{cl}\frac{|\partial B|
|\Omega|}{|B|}\,\lambda^{3/2}+\, F(l_1,l_2,\lambda)|\Omega|,
\end{equation}
where $F(l_1,l_2,\lambda)$ is shorthand notation for the lower-order terms of the left side in \eqref{R-1-rectangle-Dirichlet-bounds}

\appendix
\section{Refinements of Young's and H\"{o}lder's inequality}
In \S \ref{refinement}, we rely on an improvement of Young's inequality in order to
strengthen Kr\"oger's inequality with \eqref{Neumann-Laplacian-ev-bound-normalized2}.
Improvements of Young's inequality that are adequate for this purpose already exist in
the literature \cite{Ald,KiMa,Fur}, but we take the opportunity in this appendix to
present an efficient approach to deriving improvements to
Young's and H\"older's inequalities.

To begin, let $p>-1$. For $x\geq 0$ we define the strictly concave function $y_p(x)$ by
\begin{equation}\label{Young-fct}
  y_p(x):=(p+1)x-p-x^{p+1}.
\end{equation}
The unique critical point of $y_p(x)$ occurs at $x=1$.
Since $y_p(1)=0$, Young's inequality follows in the following formulation:
\begin{enumerate}
  \item $y_p(x)\leq 0$ for all $x\geq 0$ if $p\geq 0$
  \item $y_p(x)\geq 0$ for all $x\geq 0$ if $-1<p\leq 0$.
\end{enumerate}
Before deriving an improvement, we first note that
the case $-1<p\leq 0$ is equivalent to the case $p\geq 0$ by means of the duality
\begin{equation*}
  y_p(x)=-(p+1)y_q(z),\quad (p+1)(q+1)=1,\quad z=x^{q+1},
\end{equation*}
the fixed point of which is the trivial case $p=q=0$.
In the following we therefore only consider the case $p>0$.
Putting $\displaystyle x=a/b^{1/p}$,
defining $s=p+1$, $\displaystyle r=\frac{p+1}{p}$, such that $\displaystyle \frac1{r}+\frac1{s}=1$,
and dividing by $p+1$, we obtain the classical version of Young's inequality:
\begin{equation}\label{Young-ineq-2}
  ab-\frac{b^{r}}{r}-\frac{a^{s}}{s}\leq 0, \quad a,b\geq 0.
\end{equation}
There are basically two refinements discussed in
\cite{Ald,KiMa,Fur},
which as we shall show
follow directly from identities for the functions $y_p(x)$.
First, we consider the family of functions $f_p$ defined by
\begin{equation}\label{Young-fct-refinement1}
  f_p(x):=y_p(x)+\left(x^{(p+1)/2}-1\right)^2=2y_{(p-1)/2}(x).
\end{equation}
Clearly
\begin{enumerate}
  \item $f_p(x)\leq 0$ for all $x>0$ if $p\geq 1$,
  \item $f_1(x)=0$ for all $x>0$,
  \item $f_p(x)\geq 0$ for all $x>0$ if $0<p\leq 1$.
\end{enumerate}
When $p\geq 1$ we have $s=p+1\geq 2$, and with $\displaystyle x=a/b^{1/p}$
the refinement of Young's inequality becomes:
\begin{equation}\label{Young-ineq-2-refined1}
  ab-\frac{b^{r}}{r}-\frac{a^{s}}{s}\leq -\frac1{s}\left(a^{s/2}-b^{r/2}\right)^2, \quad a,b\geq 0,\quad s\geq 2\geq r>1.
\end{equation}
When $0\leq p\leq 1$ the inequality is reversed. Exchanging $a$ and $b$ as well as $r$ and $s$, we get:
\begin{equation}\label{Young-ineq-2-reversed1}
  ab-\frac{b^{r}}{r}-\frac{a^{s}}{s}\geq -\frac1{r}\left(a^{s/2}-b^{r/2}\right)^2, \quad a,b\geq 0,\quad s\geq 2\geq r>1.
\end{equation}
Another refinement follows from considering the
family of functions $g_p$ defined by
\begin{equation}\label{Young-fct-refinement2}
  g_p(x)=y_p(x)+p(x-1)^2=px^2-(p-1)x-x^{p+1}=x\,y_{p-1}(x).
\end{equation}
We observe that
\begin{enumerate}
  \item $g_p(x)\leq 0$ for all $x>0$ if $p\geq 1$,
  \item $g_1(x)=0$ for all $x>0$,
  \item $g_p(x)\geq 0$ for all $x>0$ if $0<p\leq 1$.
\end{enumerate}
When $p\geq 1$ we have $s=p+1\geq 2$, and with $\displaystyle x=a/b^{1/p}$ the refinement of Young's inequality becomes:
\begin{equation}\label{Young-ineq-2-refined2}
  ab-\frac{b^{r}}{r}-\frac{a^{s}}{s}\leq -\frac1{r}\left(a-b^{r-1}\right)^2b^{2-r}, \quad a,b\geq 0,\quad s\geq 2\geq r>1.
\end{equation}
When $0\leq p\leq 1$ we find a reversed inequality. Exchanging $a$ and $b$ as well as $r$ and $s$, we obtain:
\begin{equation}\label{Young-ineq-2-reversed2}
  ab-\frac{b^{r}}{r}-\frac{a^{s}}{s}\geq -\frac1{s}\left(b-a^{s-1}\right)^2a^{2-s}, \quad a,b\geq 0,\quad s\geq 2\geq r>1.
\end{equation}
Although we do not use it in this paper, we further note that
refinements of H\"{o}lder's inequality,
cf.\ \cite{CaFrLi},
are easily obtained from the inequalities above
as follows:

Let $M$ be a measure space and $a\in L^s(M)$, $b\in L^r(M)$ such that $\|a\|_s=\|b\|_r=1$ where $r^{-1}+s^{-1}=1$, $s\geq 2\geq r>1$ and $\|\cdot\|_p$ denotes the usual norm in $L^p(M)$. Then by integrating the pointwise inequalities \eqref{Young-ineq-2-refined1} and \eqref{Young-ineq-2-reversed1},
\begin{equation}\label{Hoelder-refined-1a}
  1-\frac1{r}\int\left(|a|^{s/2}-|b|^{r/2}\right)^2\leq\int |ab|\leq 1-\frac1{s}\int\left(|a|^{s/2}-|b|^{r/2}\right)^2,
\end{equation}
with equality if and only if $|a|^{s}=|b|^{r}$ pointwise almost everywhere. We also may directly make the replacements $a\rightarrow t^{-1}a$, $b \rightarrow tb$ in \eqref{Young-ineq-2-refined1} and \eqref{Young-ineq-2-reversed1} and after integration optimize with respect to $t$. This yields the slightly improved inequalities:
\begin{equation}\label{Hoelder-refined-1b}
  \left(1-\frac1{2}\int\left(|a|^{s/2}-|b|^{r/2}\right)^2\right)^{2/r}\leq\int |ab|\leq \left(1-\frac1{2}\int\left(|a|^{s/2}-|b|^{r/2}\right)^2\right)^{2/s}.
\end{equation}
When integrating the pointwise inequalities \eqref{Young-ineq-2-refined2} and \eqref{Young-ineq-2-reversed2}:
\begin{equation}\label{Hoelder-refined-1c}
  1-\frac1{s}\int\left(|b|-|a|^{s-1}\right)^2|a|^{2-s}\leq\int |ab|\leq 1-\frac1{r}\int\left(|a|-|b|^{r-1}\right)^2|b|^{2-r},
\end{equation}
with equality if and only if $|a|^{s}=|b|^{r}$ pointwise almost everywhere.

% ----------------------------------------------------------------
\bibliographystyle{amsplain}
\bibliography{references}

\end{document}